\documentclass[]{article}

\usepackage{graphicx}
\usepackage{amsmath}
\usepackage{amssymb}
\usepackage{amsthm}
\usepackage{pxfonts}
\usepackage{enumerate}
\usepackage{color}
\usepackage{mathdots}
\usepackage{sectsty}
\usepackage[hidelinks]{hyperref}

\sectionfont{\scshape\centering\fontsize{11}{14}\selectfont}
\subsectionfont{\scshape\fontsize{11}{14}\selectfont}
\usepackage{fancyhdr}

\newcommand\shorttitle{On a gateway between the Laguerre process and dynamics on partitions}
\newcommand\authors{Theodoros Assiotis}

\newtheorem{thm}{Theorem}[section]

\newtheorem{defn}[thm]{Definition}
\newtheorem{rmk}[thm]{Remark}
\newtheorem{prop}[thm]{Proposition}

\title{\large \bf ON A GATEWAY BETWEEN THE LAGUERRE PROCESS AND DYNAMICS ON PARTITIONS}
\author{\small THEODOROS ASSIOTIS}
\date{}

\begin{document}

\maketitle

\begin{abstract}
Probability measures and stochastic dynamics on matrices and on partitions are related by standard, albeit technical, discrete to continuous scaling limits. In this paper we provide exact relations, that go in both directions, between the eigenvalues of the Laguerre process and certain distinguished dynamics on partitions. This is done by generalizing to the multidimensional setting recent results of Miclo and Patie on linear one-dimensional diffusions and birth and death chains. As a corollary, we obtain an exact relation between the Laguerre and Meixner ensembles. Finally, we explain the deep connections with the Young bouquet and the z-measures on partitions.
\end{abstract}

\tableofcontents

\section{Introduction and results}

\subsection{Informal introduction}
There has been a phenomenal amount of activity around the study of random matrices and random partitions in recent decades, see for example \cite{BorodinOkounkovOlshanski}, \cite{RandomMatricesAndRandomPermutations}, \cite{SchurMeasures}, \cite{JohanssonPlancherel}, \cite{JohanssonShapeFluctuations}, \cite{JohanssonTilings}, \cite{Kerov1}, \cite{Kerov2}, \cite{BorodinOlshanskiErgodic}, \cite{BorodinOlshanskiZMeasures}, \cite{MacdonaldProcesses} \cite{Forrester}, \cite{Biane} and the references therein. Although, from the outset the study of probability measures on matrices and on partitions might not seem directly related, the mathematical tools behind it are rather similar. Most importantly, in certain scaling limits as the 'size' (we will be precise about what this means in the next subsections) $N$ of the matrix and partition go to infinity the same universal structures appear. 

However, for fixed finite sizes $N$ the connections are much less clear, other than through the rather intuitive discrete to continuous scaling limits. This is even more so, when one introduces a time variable and considers stochastic dynamics, in which case even such an intuitive scaling limit can be quite a technical challenge to establish. In the setting of one-dimensional diffusions and birth and death chains, Feller's classic paper \cite{Feller} was the first to provide a rigorous instance of such a 'diffusion approximation' result.

In this paper we prove a number of exact relations in Theorems \ref{MainResult} and \ref{MainResultStationary}, that we call gateways (borrowing the terminology from \cite{MicloPatie}) since they go in both directions (from continuous to discrete and vice versa) between the eigenvalues of the Laguerre process (and its stationary analogue) on non-negative definite Hermitian matrices and certain distinguished dynamics on partitions. As a corollary, we easily obtain in Proposition \ref{ExactRelationLaguerreMeixner} an identity between the much-studied Laguerre and Meixner ensembles.

These exact relations are in the form of intertwinings between Markov semigroups. Intertwinings have been ubiquitous in the probabilistic literature in many different contexts. Among the highlights are the seminal work of Rogers and Pitman on Markov functions \cite{RogersPitman}, Carmona-Petit-Yor's study of the beta-gamma algebra \cite{BetaGamma} and certainly Diaconis and Fill's application of intertwining relations to the study of convergence to equilibrium for Markov chains \cite{DiaconisFill}. In recent years, intertwinings have been used extensively in the field of integrable probability for a range of different problems, see for example \cite{BorodinOlshanskiThoma}, \cite{BorodinOlshanski}, \cite{PalShkolnikov}, \cite{O ConnellTams}, \cite{HuaPickrell}, \cite{InterlacingDiffusions}, \cite{Cuenca}, \cite{Cerenzia}, \cite{CerenziaKuan}, \cite{RandomGrowthKarlinMcGregor}, \cite{Sun} and the references therein.

Our work builds upon and generalizes to the multidimensional setting the recent results of Miclo and Patie \cite{MicloPatie} on linear one-dimensional diffusions and birth and death chains (one of these one-dimensional results is also due to Borodin and Olshanski in \cite{BorodinOlshanskiThoma} by a different method, see Remark \ref{RemarkBOproof} and Section 3 for more details).

Finally, we should mention that our argument relies in a key way on the underlying determinantal structure, in the form of the celebrated Karlin-McGregor formula. This allows us to lift, modulo some technical work, the one-dimensional relations to the multidimensional setting in a rather neat way.

The outline of the rest of the paper is as follows. In the rest of the introduction we give the necessary background and state our results precisely. In Section 2 we give the proofs. Finally, in Section 3 we explain the deep connections between this paper and a series of works by Borodin and Olshanski on the Young bouquet and Markov processes for the z-measures on partitions \cite{BorodinOlshanskiYoungBouquet}, \cite{BorodinOlshanskiThoma}, \cite{BorodinOlshanskiZMeasures}, \cite{BorodinOlshanskiMarkovonPartitions}, \cite{OlshanskiLaguerreDiffusions}, \cite{OlshanskiSymmetricFunctions}.

\subsection{Setup and first set of results}
\subsubsection{The Laguerre process and its eigenvalues}
Let $H(N)$ be the space of $N \times N$ Hermitian matrices and $H_+(N)$ the subspace of non-negative definite ones.
 
We now introduce the Laguerre process on non-negative definite Hermitian matrices, depending on a parameter $\beta>0$. The analogous process on real symmetric matrices was first considered by Bru in \cite{Wishart} under the name of Wishart process. The Hermitian case that we will be concerned with was then introduced in \cite{O Connell} and further studied in \cite{Demni}.

 Let $\left(\boldsymbol{W}_t;t \ge0\right)$ be an $N\times N$ complex Brownian matrix. More precisely:
\begin{align*}
\left[\boldsymbol{W}_t\right]_{ij}=\gamma_{ij}(t)+\mathrm{i}\tilde{\gamma}_{ij}(t)
\end{align*}
for $\{\gamma_{ij}(\cdot) \}_{i,j=1}^N,\{\tilde{\gamma}_{ij}(\cdot) \}_{i,j=1}^N$ independent standard real Brownian motions. Then, the Laguerre process $\left(\boldsymbol{X}_t;t\ge0\right)$ (we suppress dependence on $\beta$) is given by the solution to the matrix stochastic differential equation (SDE):
\begin{align}
d\boldsymbol{X}_t=d\boldsymbol{W}_t\sqrt{\frac{\boldsymbol{X}_t^2}{2}}+\sqrt{\frac{\boldsymbol{X}_t^2}{2}}d\boldsymbol{W}_t^{\dagger}+\left[\beta+(N-1)\right]\boldsymbol{I}dt.
\end{align}
Here, $\mathbf{H}^{\dagger}$ denotes the complex conjugate of a matrix $\mathbf{H}$ and $\mathbf{I}$ is the identity matrix.

We will be interested in the evolution of the eigenvalues of $\left(\boldsymbol{X}_t;t\ge0\right)$. We first define, the Weyl chamber $W^N_{c,+}$ with non-negative (the subscript $c$ stands for continuous, we will also introduce a discrete version later on) coordinates by:
\begin{align*}
W^N_{c,+}&=\{ x=(x_1,\cdots,x_N)\in \mathbb{R}_+^N: x_1\le x_2  \le \cdots \le x_N \},
\end{align*}
where $\mathbb{R}_+=[0,\infty)$. Write $\mathsf{eval}_N:H_+(N) \to W^{N}_{c,+}$ for the map sending a non-negative definite matrix $\mathbf{H}$ to its ordered eigenvalues $x=(x_1,\cdots,x_N)\in W^{N}_{c,+}$. Then, 
$\left(\mathsf{eval}_N\left(\mathbf{X}_t\right);t\ge0\right)=\left(x_1(t),\cdots, x_N(t);t \ge 0\right)$ follows the stochastic differential system in $W^{N}_{c,+}$,
\begin{align}\label{RepulsiveSDE}
dx_i(t)=\sqrt{2x_i(t)}dw_i(t)+\left(\beta+\sum_{j\ne i}^{}\frac{2x_i(t)}{x_i(t)-x_j(t)}\right)dt, \ 1\le i \le N,
\end{align}
for some independent standard real Brownian motions $\{w_i(\cdot)\}_{i=1}^N$, see \cite{MatrixYamadaWatanabe}, \cite{Graczyk}, \cite{O Connell}. This system of SDEs has a unique strong solution with no collisions or explosions even when started from a point with coinciding coordinates, see \cite{Graczyk}. In particular, for any initial condition $x(0)=\left(x_1(0),\cdots, x_N(0)\right)\in W^N_{c,+}$:
\begin{align}\label{nonintersecting}
x_1(t)<x_2(t)<\cdots<x_N(t), \ \forall t>0, \textnormal{ almost surely}.
\end{align}

It is a remarkable fact, first observed in \cite{O Connell}, that this system of SDEs is exactly solvable, in a way that we now describe. First, write $q_t^{(\beta)}(x,y)$ for the transition density with respect to Lebesgue measure of the one dimensional diffusion process in $(0,\infty)$ with generator (a version of the squared Bessel process, see \cite{SurveyBessel}):
\begin{align}
G^{(\beta)}=x\frac{d^2}{dx^2}+\beta \frac{d}{dx}.
\end{align}
Observe that, this linear diffusion is the special case $N=1$ of (\ref{RepulsiveSDE}). We note that $q_t^{(\beta)}(x,y)$ has a well-known explicit expression in terms of Bessel functions (see for example \cite{SurveyBessel}), that we shall not need here though.

Then, as proven in \cite{O Connell}, \cite{Demni} the solution of the system of SDEs (\ref{RepulsiveSDE}) can be realized as $N$ independent copies of $G^{(\beta)}$-diffusions conditioned to never intersect via a Doob h-transform. The corresponding transition kernel is then given by the Doob h-transformed Karlin-McGregor determinant \cite{Karlin} defined by,
\begin{align}\label{KarlinMcGregorContinuous1}
q^{N,(\beta)}_t(x,dy)=\frac{\Delta_N(y)}{\Delta_N(x)}\det \left(q_t^{(\beta)}(x_i,y_j)\right)_{i,j=1}^Ndy_1\cdots dy_N, \ \ \forall (t,x,y) \in (0,\infty) \times \mathring{W}^N_{c,+} \times W^N_{c,+}.
\end{align}
Here and throughout the paper we write
\begin{align*}
\Delta_N(x)=\det\left(x_i^{j-1}\right)^N_{i,j=1}=\prod_{1 \le i <j \le N}^{}(x_j-x_i)
\end{align*}
for the Vandermonde determinant. Also, $\mathring{W}^N_{c,+}$ denotes the interior of $W^N_{c,+}$, namely when none of the coordinates coincide (the fact that this definition can be continuously extended to the boundary $\partial W^N_{c,+}$ is part of Proposition \ref{FellerSemigroups}).

Let $\left(Q_t^{N,(\beta)}\right)_{t\ge 0}$ denote the corresponding semigroup with transition kernel (\ref{KarlinMcGregorContinuous1}), associated to (\ref{RepulsiveSDE}). Finally, observe that for all $x \in \mathring{W}^N_{c,+}$ and $t>0$, the measure $q_t^{N,(\beta)}(x,\cdot)$ is supported on $\mathring{W}^N_{c,+}$ (in fact, due to (\ref{nonintersecting}), this holds for any $x \in W^N_{c,+}$).

\subsubsection{The discrete dynamics: non-intersecting linear birth and death chains}
We first need some background on partitions. A finite non-increasing sequence of non-negative integers $\lambda=(\lambda_1\ge \lambda_2\ge \cdots \ge 0)$ is called a partition. It is well known that partitions can be identified with Young diagrams, the set of which we denote by $\mathbb{Y}$. We write $|\lambda|=\sum_{i}^{}\lambda_i$ (equivalently the number of boxes in the Young diagram corresponding to $\lambda$) and also $l(\lambda)$ for the length of a partition, namely the largest index $k$ such that $\lambda_k>0$ (equivalently the number of rows in the corresponding diagram). 

Let $\mathbb{Y}(N)$ denote the set of all Young diagrams with at most $N$ rows, equivalently partitions $\lambda$ such that $l(\lambda) \le N$ (not to be confused with the set of Young diagrams with exactly $N$ boxes, usually denoted by $\mathbb{Y}_N$). Moreover, define the discrete non-negative Weyl chamber:
\begin{align*}
W^N_{d,+}&=\{ x=(x_1,\cdots,x_N)\in \mathbb{Z}_+^N: x_1< x_2 <  \cdots < x_N \},
\end{align*}
where $\mathbb{Z}_+=\{0,1,2,\dots\}$. Then, it is well known that we have the following bijection between $\mathbb{Y}(N)$ and $W^N_{d,+}$:
\begin{align*}
\lambda=\left(\lambda_1,\cdots,\lambda_N\right)\in \mathbb{Y}(N) \mapsto y=\left(\lambda_N,\lambda_{N-1}+1,\cdots, \lambda_2+N-2,\lambda_1+N-1\right) \in W^N_{d,+}.
\end{align*}
Thus, from now on we can and will only consider $W^N_{d,+}$.

We are ready to introduce our discrete dynamics. Consider the following birth and death chain, namely a Markov chain in continuous time on $\mathbb{Z}_+$ that moves with jumps of size $\pm 1$, with rate when at site $n$ of jumping to the right $n+\beta$ and for moving to the left $n$. We write $\nabla_+$ and $\nabla_-$ for the forward and backward discrete derivatives respectively:
\begin{align*}
\nabla_+g(n)=g(n+1)-g(n) \ , \ \nabla_-g(n)=g(n-1)-g(n).
\end{align*}
Then, the generator $\mathfrak{G}^{(\beta)}$ of the birth and death chain we are considering is given by:
\begin{align}
\mathfrak{G}^{(\beta)}=(n+\beta) \nabla_++n\nabla_-.
\end{align}
We write $\mathfrak{q}^{(\beta)}_t(x,y)$ for its transition density.

Now, consider $N$ identical copies of this birth and death chain conditioned to never intersect. The transition kernel of this Markov process is then given by the Doob transformed Karlin-McGregor semigroup \cite{Karlin}, \cite{KarlinMcGregorCoincidence}, see Chapter 6 of \cite{Doumerc} where this specific example was first studied:
\begin{align}
\mathfrak{q}^{N,(\beta)}_t(x,y)=\frac{\Delta_N(y)}{\Delta_N(x)}\det \left(\mathfrak{q}_t^{(\beta)}(x_i,y_j)\right)_{i,j=1}^N, \ t>0, x,y\in W^N_{d,+}.
\end{align}
We denote by $\left(\mathfrak{Q}_t^{N,(\beta)}\right)_{t\ge 0}$ the semigroup on $W^{N}_{d,+}$ with transition kernel $\mathfrak{q}^{N,(\beta)}_t(x,y)$.

\subsubsection{Intertwinings}
We now introduce the exact link between the continuous and discrete dynamics. We first need an abstract definition. Let $\mathsf{X}$ and $\mathsf{Y}$ be two measurable spaces. A Markov kernel $\mathsf{\Lambda}$ from $\mathsf{X}$ to $\mathsf{Y}$ is a function $\mathsf{\Lambda}(\mathsf{x},\mathsf{A})$, where the first argument $\mathsf{x}$ ranges over $\mathsf{X}$, while the second argument is a measurable subset of $\mathsf{Y}$ so that:
\begin{itemize}
\item For fixed $\mathsf{A}$, $\mathsf{\Lambda}(\cdot,\mathsf{A})$ is a measurable function on $\mathsf{X}$.
\item For $\mathsf{x}$ fixed, $\mathsf{\Lambda}(\mathsf{x},\cdot)$ is a probability measure on $\mathsf{Y}$.
\end{itemize}
We then consider the following Markov (as to be shown in Proposition \ref{FellerKernels} below) kernel $\Lambda_N$ from $W^{N}_{c,+}$ to $W^N_{d,+}$ defined by (its density with respect to counting measure), for $x \in \mathring{W}^N_{c,+}$:
\begin{align}
\Lambda_N\left(x,y\right)=\frac{\Delta_N(y)}{\Delta_N(x)}\det \left(\frac{x_i^{y_j}e^{-x_i}}{y_j!}\right)_{i,j=1}^N, y \in W^N_{d,+}.
\end{align}
It is not hard to see that $\Lambda_N(x,y)$ can be continuously extended to $x \in \partial W^{N}_{c,+}$ since the singularities coming from $1/\Delta_N(x)$ at $x_i=x_j$, $i\neq j$ are cancelled out by $\det \left(\frac{x_i^{y_j}e^{-x_i}}{y_j!}\right)_{i,j=1}^N$ which vanishes at those hyperplanes (see also the proof of Proposition \ref{FellerKernels}). This Markov kernel is moreover intimately related to the Young bouquet as explained in Section \ref{ConnectiontoYoungBouquetSection}. 

\begin{rmk}
The determinant
\begin{align*}
\det \left(\frac{x_i^{y_j}e^{-x_i}}{y_j!}\right)_{i,j=1}^N, \ x \in W^N_{c,+}, y \in W^N_{d,+},
\end{align*}
has an interesting probabilistic interpretation in terms of non-intersecting paths of Poisson processes starting at different times, explained in detail in Remark 3.2 of \cite{JohanssonTilings}.
\end{rmk}

We also consider a link in the opposite direction, namely the Markov (as to be shown below) kernel $\Lambda_{N,\beta}^*$ from $W^N_{d,+}$ to $W^N_{c,+}$, defined by for $y\in W^N_{d,+}$:
\begin{align}
\Lambda^*_{N,\beta}\left(y,dx\right)=\frac{\Delta_N(x)}{\Delta_N(y)}\det \left(\frac{x_i^{y_j+\beta-1}e^{-x_i}}{\Gamma(y_j+\beta)}\right)_{i,j=1}^Ndx_1\cdots dx_N.
\end{align}

Observe that, $\Lambda_{N,\beta}^*$ depends on the parameter $\beta$, unlike $\Lambda_N$. Moreover, note that for all $y\in W^N_{d,+}$, the measure $\Lambda_{N,\beta}^*(y,\cdot)$ is supported on $\mathring{W}^N_{c,+}$.

Let $C_0\left(W^N_{c,+}\right), C_0\left(W^N_{d,+}\right)$ denote the spaces of continuous functions vanishing at infinity on $W^N_{c,+}$ and $W^N_{d,+}$ respectively. Then, $\Lambda_N, \Lambda_{N,\beta}^*$ have the Feller-Markov property:
\begin{prop}\label{FellerKernels}
Let $\beta \ge 0$ and $N\ge 1$. The kernels $\Lambda_N$ and $\Lambda_{N,\beta}^*$ are Feller-Markov. Namely, for all $x \in W^{N}_{c,+}$ and $y\in W^N_{d,+}$,
\begin{align*}
\Lambda_N\left(x,\cdot\right) \textnormal{ and } \ \Lambda_{N,\beta}^*(y,\cdot)
\end{align*}
are probability measures on $W^{N}_{d,+}$ and $W^{N}_{c,+}$ respectively and moreover for all $f\in C_{0}\left(W^{N}_{c,+}\right)$ and $\mathfrak{f}\in C_{0}\left(W^{N}_{d,+}\right)$:
\begin{align*}
\Lambda_N\mathfrak{f}\in C_{0}\left(W^{N}_{c,+}\right) \textnormal{ and } \ \Lambda_Nf\in C_{0}\left(W^{N}_{d,+}\right).
\end{align*}
\end{prop}

Similarly, the semigroups $\left(Q_t^{N,(\beta)}\right)_{t\ge 0}$ and $\left(\mathfrak{Q}_t^{N,(\beta)}\right)_{t\ge 0}$ are also Feller:
\begin{prop}\label{FellerSemigroups}
Let $\beta> 0$ and $N\ge 1$. The semigroups $\left(Q_t^{N,(\beta)}\right)_{t\ge 0}$ and $\left(\mathfrak{Q}_t^{N,(\beta)}\right)_{t\ge 0}$ are Feller-Markov: For all $f\in C_{0}\left(W^{N}_{c,+}\right)$ and $\mathfrak{f}\in C_{0}\left(W^{N}_{d,+}\right)$:
\begin{align*}
\lim_{t\to 0}Q_t^{N,(\beta)}f =f,\ Q_t^{N,(\beta)}f\in C_{0}\left(W^{N}_{c,+}\right),\\
\lim_{t\to 0}\mathfrak{Q}_t^{N,(\beta)}\mathfrak{f} =\mathfrak{f},\ \mathfrak{Q}_t^{N,(\beta)}\mathfrak{f}\in C_{0}\left(W^{N}_{d,+}\right).
\end{align*}
\end{prop}

We finally arrive at our main result.

\begin{thm}\label{MainResult}
Let $\beta>0$. For all $N\ge1, t \ge 0$ we have the following equality between Feller-Markov kernels:
\begin{align}
Q_t^{N,(\beta)}\Lambda_N&=\Lambda_N\mathfrak{Q}_t^{N,(\beta)}\label{MainIntertwining1},\\
\mathfrak{Q}_t^{N,(\beta)}\Lambda^*_{N,\beta}&=\Lambda^*_{N,\beta}Q_t^{N,(\beta)}\label{MainIntertwining2}.
\end{align} 
In particular, for all $f\in C_{0}\left(W^{N}_{c,+}\right)$ and $\mathfrak{f}\in C_{0}\left(W^{N}_{d,+}\right)$:
\begin{align*}
Q_t^{N,(\beta)}\Lambda_N\mathfrak{f}=\Lambda_N\mathfrak{Q}_t^{N,(\beta)}\mathfrak{f}\ \ , \ \ \mathfrak{Q}_t^{N,(\beta)}\Lambda^*_{N,\beta}f=\Lambda^*_{N,\beta}Q_t^{N,(\beta)}f.
\end{align*}
\end{thm}

\begin{rmk}
We should emphasize that we do not give an independent proof of the Miclo-Patie result \cite{MicloPatie}, which is the case $N=1$, but rather (assume it and) use it as a key ingredient in our argument for $N\ge 2$ which is the contribution of the present paper.
\end{rmk}

\begin{rmk}
We can see from Theorem \ref{MainResult} that $\Lambda_N\Lambda_{N,\beta}^*$ commutes with $Q_t^{N,(\beta)}$. In fact, the following relation is true:
\begin{align}
\Lambda_N\Lambda_{N,\beta}^*=Q_1^{N,(\beta)}.
\end{align}
Similarly, we also have:
\begin{align}
\Lambda_{N,\beta}^*\Lambda_N=\mathfrak{Q}_1^{N,(\beta)}.
\end{align}
Both of these relations can be proven in the same fashion as Theorem \ref{MainResult}, making use of the $N=1$ cases, Proposition 13 and 14 of \cite{MicloPatie}. The details are left to the reader.
\end{rmk}

\subsection{The stationary case}

\subsubsection{The stationary dynamics}

We will now consider the stationary analogues of the results above. As before, throughout this subsection the parameter $\beta>0$. Write $k_t^{(\beta)}(x,y)$ for the transition density with respect to Lebesgue measure of the one dimensional diffusion process in $(0,\infty)$ with generator
\begin{align}
L^{(\beta)}=x\frac{d^2}{dx^2}+\left(\beta-x\right) \frac{d}{dx}.
\end{align}
This is the stationary analogue of $G^{(\beta)}$ (see \cite{SurveyBessel}, \cite{MicloPatie}). It is reversible (see for example \cite{MicloPatie}) with respect to the probability measure (the law of a Gamma random variable) on $(0,\infty)$:
\begin{align*}
\nu_{\beta}(dx)=\nu_{\beta}(x)\textbf{1}_{\{x\in (0,\infty)\} }dx=\frac{1}{\Gamma(\beta)}x^{\beta-1}\exp(-x)\textbf{1}_{\{x\in (0,\infty)\} }dx.
\end{align*}
We can consider the unique strong solution to the following system of non-colliding and non-exploding SDEs in $W^N_{c,+}$, see \cite{Graczyk}: 
\begin{align}\label{StationarySDEs}
dx_i(t)=\sqrt{2x_i(t)}dw_i(t)+\left(\beta-x_i(t)+\sum_{j\ne i}^{}\frac{2x_i(t)}{x_i(t)-x_j(t)}\right)dt, \ 1\le i \le N,
\end{align}
for some independent standard real Brownian motions $\{w_i\}_{i=1}^N$. As before, the system of SDEs (\ref{StationarySDEs}) is exactly solvable in terms of a single $L^{(\beta)}$-diffusion. More precisely, the transition kernel of the solution of these SDEs is given by a Doob h-transformed Karlin-McGregor semigroup:
\begin{align}
k^{N,(\beta)}_t(x,dy)=e^{\frac{N(N-1)}{2}t}\frac{\Delta_N(y)}{\Delta_N(x)}\det \left(k_t^{(\beta)}(x_i,y_j)\right)_{i,j=1}^Ndy_1\cdots dy_N, \ \ \forall (t,x,y) \in (0,\infty) \times \mathring{W}^N_{c,+} \times W^N_{c,+}.
\end{align}
Write $\left(K_t^{N,(\beta)}\right)_{t\ge 0}$ for the Markov semigroup on $W^N_{c,+}$ with transition kernel $k^{N,(\beta)}_t(x,dy)$.

We now introduce the stationary version of the discrete dynamics. Consider the following birth and death chain, with rate, when at site $n$, of jumping to the right $\sigma\left(n+\beta\right)$ and for jumping to the left $(\sigma+1)n$. Here, the parameter $\sigma>0$. The generator $\mathfrak{L}^{(\beta),\sigma}$ of this birth and death chain is then given by:
\begin{align}
\mathfrak{L}^{(\beta),\sigma}=\sigma (n+\beta) \nabla_++(\sigma+1)n\nabla_-.
\end{align}
Denote by $\mathfrak{k}_t^{(\beta),\sigma}$ its transition density. Moreover, we note that (see for example \cite{BorodinOlshanskiMarkovonPartitions}, \cite{MicloPatie}) this chain is reversible with respect to the negative binomial distribution $\eta_{\beta,\sigma}(\cdot)$ on $\mathbb{Z}_+$:
\begin{align*}
\eta_{\beta,\sigma}(n)=\sigma^n(1+\sigma)^{-n-\beta}\binom{n+\beta-1}{n}.
\end{align*}
Now, consider the corresponding Doob h-transformed Karlin-McGregor determinant given by: 
\begin{align}
\mathfrak{k}^{N,(\beta),\sigma}_t(x,y)=e^{\frac{N(N-1)}{2}t}\frac{\Delta_N(y)}{\Delta_N(x)}\det \left(\mathfrak{k}_t^{(\beta),\sigma}(x_i,y_j)\right)_{i,j=1}^N, \ t>0, x,y\in W^N_{d,+}.
\end{align}
This first appeared in Section 3 of \cite{BorodinOlshanskiMarkovonPartitions}, see also Section 6 of that paper for the interpretation as $N$ independent copies of an $\mathfrak{L}^{(\beta),\sigma}$-chain conditioned to never intersect. Moreover, we denote by $\left(\mathfrak{K}_t^{N,(\beta)}\right)_{t\ge 0}$ the semigroup on $W^{N}_{d,+}$ with transition kernel $\mathfrak{k}^{N,(\beta)}_t(x,y)$.

Finally, we introduce the following Markov kernel $\Lambda_{N,\sigma}$ from $W^N_{c,+}$ to $W^N_{d,+}$ defined by, for $x \in \mathring{W}^N_{c,+}$ (as before it can be continuously extended to $x \in \partial W^N_{c,+}$):
\begin{align}\label{MarkovKernel}
\Lambda_{N,\sigma}\left(x,y\right)=\Lambda_N(\sigma x,y)=\sigma^{-\frac{N(N-1)}{2}}\frac{\Delta_N(y)}{\Delta_N(x)}\det \left(\frac{(\sigma x_i)^{y_j}e^{-\sigma x_i}}{y_j!}\right)_{i,j=1}^N, \ y \in W^N_{d,+}.
\end{align}
Observe that, $\Lambda_N$ is the special case $\Lambda_{N,1}$ with $\sigma=1$. For the connection to the Young bouquet, see Section \ref{ConnectiontoYoungBouquetSection}.

As before, we have the Feller property.

\begin{prop}\label{FellerStationary}
Let $\sigma>0$, $\beta>0$ and $N\ge 1$. The kernel $\Lambda_{N,\sigma}$ and the semigroups $\left(K_t^{N,(\beta)}\right)_{t\ge 0}$ and $\left(\mathfrak{K}_t^{N,(\beta),\sigma}\right)_{t\ge 0}$ are Feller-Markov.
\end{prop}

Finally, we have the following stationary analogue of Theorem \ref{MainResult}.
\begin{thm}\label{MainResultStationary}
Let $\sigma>0$ and $\beta>0$. For all $N\ge1, t \ge 0$ we have the following equality between Feller-Markov kernels:
\begin{align}\label{stationaryintertwining}
K_t^{N,(\beta)}\Lambda_{N,\sigma}=\Lambda_{N,\sigma}\mathfrak{K}_t^{N,(\beta),\sigma}.
\end{align}
\end{thm}

\begin{rmk}\label{RemarkBOproof}
The case $N=1$ is proven in \cite{MicloPatie}. In fact, a proof by different methods first appeared in Section 6 of \cite{BorodinOlshanskiThoma} as part of a more general scheme. Again, we do not give an independent proof of this case but rather use it as a key ingredient.
\end{rmk}

\subsubsection{The stationary measures: a relation between the Laguerre and Meixner ensembles}
For $\beta>0$, consider the Laguerre ensemble (or complex Wishart probability measure), see \cite{WishartOriginal}, \cite{Forrester}, on $N \times N$ Hermitian matrices, supported on $H_+(N)$:
\begin{align*}
\mathsf{M}^{(\beta),N}(d\mathbf{H})=\textnormal{const}_{\beta, N}\det\left(\mathbf{H}\right)^{\beta-1} e^{-\textnormal{Tr}\mathbf{H}}\mathbf{1}_{\left\{\mathbf{H} \in H_+(N)\right\}}d\mathbf{H}
\end{align*}
where $d\mathbf{H}$ denotes Lebesgue measure on $H(N)$:
\begin{align*}
d\mathbf{H}=\prod_{j=1}^{N}d\mathbf{H}_{jj} \prod_{1 \le j <k \le N}^{}d \Re \mathbf{H}_{jk}d \Im \mathbf{H}_{jk}.
\end{align*}
Then, by Weyl's integration formula the induced probability measure on eigenvalues on $W^{N}_{c,+}$ is given by:
\begin{align}
\nu_{\beta}^N(dx)=\left(\mathsf{eval}_N\right)_*\mathsf{M}^{(\beta),N}(dx)=const_{N, \beta} \Delta_N(x)^2\prod_{i=1}^{N}\nu_{\beta}(dx_i).
\end{align}
Finally, we define the Meixner ensemble to be the following probability measure on $W^N_{d,+}$, where the parameters $\sigma,\beta>0$:
\begin{align}
\eta_{\beta, \sigma}^N(\lambda)=const_{N,\beta,\sigma} \Delta_N(\lambda)^2 \prod_{i=1}^{N}\eta_{\beta,\sigma}(\lambda_i), \ \lambda \in W^N_{d,+}.
\end{align}
This appears in problems of last passage percolation, see \cite{JohanssonMarkovChain},\cite{JohanssonShapeFluctuations}, \cite{JohanssonTilings}, \cite{JohanssonPlancherel} and is also a special case of the distinguished z-measures on partitions, see Section \ref*{zMeasuresSubsection} and also the original papers \cite{BorodinOlshanskiZMeasures}, \cite{BorodinOlshanskiMeixner}.

The following proposition is well-known but we also give a proof for completeness.
\begin{prop}\label{StationaryMeasuresProposition}
Let $\sigma>0$, $\beta>0$ and $N\ge 1$. Then, $\nu_{\beta}^N$ is invariant for the semigroup $\left(K_t^{N,(\beta)}\right)_{t\ge 0}$. Moreover, $\eta^N_{\beta,\sigma}$ is the unique invariant measure of $\left(\mathfrak{K}_t^{N,(\beta),\sigma}\right)_{t\ge 0}$.
\end{prop}

\begin{rmk}
In fact, $\nu_{\beta}^N$ is the unique invariant measure of $K_t^{N,(\beta)}$ but we shall not need this here.
\end{rmk}
We finally, obtain the following exact relation between the Laguerre and Meixner ensembles.

\begin{prop}\label{ExactRelationLaguerreMeixner} Let $\sigma>0$ and $\beta>0$. For all $\lambda \in W^N_{d,+}$ we have:
\begin{align}
\left[\nu_{\beta}^N\Lambda_{N,\sigma}\right](\lambda)=\eta_{\beta, \sigma}^N(\lambda).
\end{align}
\end{prop}

\begin{proof}
Apply $\nu_{\beta}^N$ to both sides of (\ref{stationaryintertwining}):
\begin{align*}
\nu_{\beta}^NK_t^{N,(\beta)}\Lambda_{N,\sigma}=\nu_{\beta}^N\Lambda_{N,\sigma}\mathfrak{K}_t^{N,(\beta),\sigma}, \ \forall t \ge 0.
\end{align*}
By invariance of $\nu_{\beta}^N$ for $K_t^{N,(\beta)}$:
\begin{align*}
\nu_{\beta}^N\Lambda_{N,\sigma}=\nu_{\beta}^N\Lambda_{N,\sigma}\mathfrak{K}_t^{N,(\beta),\sigma}, \ \forall t \ge 0.
\end{align*}
By uniqueness of the invariant measure of $\mathfrak{K}_t^{N,(\beta),\sigma}$ we obtain the statement of the proposition.
\end{proof}

\paragraph{Acknowledgements.} I would like to thank an anonymous referee for a careful reading of the paper and some very useful comments and suggestions. Research supported by ERC Advanced Grant  740900 (LogCorRM).

\section{Proofs}

We first prove Theorem \ref{MainResult} assuming Propositions \ref{FellerKernels} and \ref{FellerSemigroups}.

\begin{proof}[Proof of Theorem \ref{MainResult}]
We first prove relation (\ref{MainIntertwining1}). As already mentioned in the introduction, the key ingredient is the $N=1$ case of the theorem, proven as Theorem 1 in \cite{MicloPatie}, that we recall in our notation as follows, for $t>0, x \in \mathbb{R}_+, y \in \mathbb{Z}_+$:
\begin{align}\label{1dintertwining1}
\int_{0}^{\infty}q_t^{(\beta)}(x,z)\frac{z^ye^{-z}}{y!}dz=\sum_{w=0}^{\infty}\frac{x^we^{-x}}{w!}\mathfrak{q}_t^{(\beta)}(w,y).
\end{align}
Let $N\ge 1$ be arbitrary. We calculate for $t>0$ and $x \in \mathring{W}^N_{c,+}$, where we use the fact that $Q_t^{N,(\beta)}(x,dz)$ is supported in $\mathring{W}^N_{c,+}$, the Andreif identity and (\ref{1dintertwining1}): 
\begin{align*}
Q_t^{N,(\beta)}\Lambda_N\left(x,y\right)&= \int_{z \in \mathring{W}^N_{c,+}}^{}\frac{\Delta_N(z)}{\Delta_N(x)}\det \left(q_t^{(\beta)}(x_i,z_j)\right)_{i,j=1}^N\frac{\Delta_N(y)}{\Delta_N(z)}\det \left(\frac{z_i^{y_j}e^{-z_i}}{y_j!}\right)_{i,j=1}^N dz_1\cdots dz_N\\
&= \frac{\Delta_N(y)}{\Delta_N(x)}\int_{z \in \mathring{W}^N_{c,+}}^{}\det \left(q_t^{(\beta)}(x_i,z_j)\right)_{i,j=1}^N\det \left(\frac{z_i^{y_j}e^{-z_i}}{y_j!}\right)_{i,j=1}^N dz_1\cdots dz_N\\
&= \frac{\Delta_N(y)}{\Delta_N(x)}\det\left[\int_{0}^{\infty}q_t^{(\beta)}(x_i,z)\frac{z^{y_j}e^{-z}}{y_j!}dz\right]_{i,j=1}^N\\
&=\frac{\Delta_N(y)}{\Delta_N(x)}\det\left[\sum_{w=0}^{\infty}\frac{x_i^we^{-x_i}}{w!}\mathfrak{q}_t^{(\beta)}(w,y_j)\right]_{i,j=1}^N.
\end{align*}
On the other hand:
\begin{align*}
\Lambda_N\mathfrak{Q}_t^{N,(\beta)}(x,y)&=\sum_{w \in W^N_{d,+}}^{}\frac{\Delta_N(w)}{\Delta_N(x)}\det \left(\frac{x_i^{w_j}e^{-x_i}}{w_j!}\right)_{i,j=1}^N\frac{\Delta_N(y)}{\Delta_N(w)}\det \left(\mathfrak{q}_t^{(\beta)}(w_i,y_j)\right)_{i,j=1}^N\\
&=\frac{\Delta_N(y)}{\Delta_N(x)}\sum_{w \in W^N_{d,+}}^{}\det \left(\frac{x_i^{w_j}e^{-x_i}}{w_j!}\right)_{i,j=1}^N\det \left(\mathfrak{q}_t^{(\beta)}(w_i,y_j)\right)_{i,j=1}^N\\
&=\frac{\Delta_N(y)}{\Delta_N(x)}\det\left[\sum_{w=0}^{\infty}\frac{x_i^we^{-x_i}}{w!}\mathfrak{q}_t^{(\beta)}(w,y_j)\right]_{i,j=1}^N.
\end{align*}
Thus, we obtain that both sides are equal for $t>0$ and $x \in \mathring{W}^N_{c,+}$. Using the Feller property of all the Markov kernels involved we extend this to:
\begin{align}
Q_t^{N,(\beta)}\Lambda_N\left(x,\cdot\right)=\Lambda_N\mathfrak{Q}_t^{N,(\beta)}(x,\cdot), \ \forall t \ge 0, \forall x \in W^N_{c,+}.
\end{align}

We now turn to relation (\ref{MainIntertwining2}). The $N=1$ case, again proven as Theorem 1 in \cite{MicloPatie}, which is as follows in our notation, for $t>0, y \in \mathbb{Z}_+, x \in \mathbb{R}_+$:
\begin{align}\label{1dintertwining2}
\sum_{w=0}^{\infty}\mathfrak{q}_t^{(\beta)}(y,w) \frac{x^{w+\beta-1}e^{-x}}{\Gamma(w+\beta)}dx=
\int_{0}^{\infty}\frac{z^{y+\beta-1}e^{-z}}{\Gamma(y+\beta)}q_t^{(\beta)}(z,x)dzdx.
\end{align}
Again, we calculate for $t>0$ using (\ref{1dintertwining2}):
\begin{align*}
\mathfrak{Q}_t^{N,(\beta)}\Lambda^*_{N,\beta}\left(y,dx\right)&=\sum_{w \in W^N_{d,+}}^{}\frac{\Delta_N(x)}{\Delta_N(z)}\det\left(\mathfrak{q}_t^{(\beta)}(y_i,w_j)\right)_{i,j=1}^N\frac{\Delta_N(z)}{\Delta_N(y)}\det \left(\frac{x_i^{w_j+\beta-1}e^{-x_i}}{\Gamma(w_j+\beta)}\right)_{i,j=1}^Ndx_1\cdots dx_N\\
&=\frac{\Delta_N(x)}{\Delta_N(y)}\sum_{w \in W^N_{d,+}}^{}\det\left(\mathfrak{q}_t^{(\beta)}(y_i,w_j)\right)_{i,j=1}^N\det \left(\frac{x_i^{w_j+\beta-1}e^{-x_i}}{\Gamma(w_j+\beta)}\right)_{i,j=1}^Ndx_1\cdots dx_N \\
&=\frac{\Delta_N(x)}{\Delta_N(y)}\det\left[
\sum_{w=0}^{\infty}\mathfrak{q}_t^{(\beta)}(y_i,w) \frac{x_j^{w+\beta-1}e^{-x_j}}{\Gamma(w+\beta)}\right]_{i,j=1}^Ndx_1\cdots dx_N \\
&=\frac{\Delta_N(x)}{\Delta_N(y)}\det\left[
\int_{0}^{\infty}\frac{z^{y_i+\beta-1}e^{-z}}{\Gamma(y_i+\beta)}q_t^{(\beta)}(z,x_j)dz\right]_{i,j=1}^Ndx_1\cdots dx_N.
\end{align*}
While on the other hand we have, since $\Lambda_{N,\beta}^*(y,\cdot)$ is supported on $\mathring{W}^N_{c,+}$:
\begin{align*}
\Lambda^*_{N,\beta}Q_t^{N,(\beta)}\left(y,dx\right)&=\frac{\Delta_N(x)}{\Delta_N(y)}\int_{z \in \mathring{W}^N_{c,+}}^{}\det \left(\frac{z_i^{y_j+\beta-1}e^{-z_i}}{\Gamma(y_j+\beta)}\right)_{i,j=1}^N\det\left(q_t^{(\beta)}(z_i,x_j)\right)_{i,j=1}^Ndz_1 \cdots dz_Ndx_1\cdots dx_N\\
&=\frac{\Delta_N(x)}{\Delta_N(y)}\det\left[
\int_{0}^{\infty}\frac{z^{y_i+\beta-1}e^{-z}}{\Gamma(y_i+\beta)}q_t^{(\beta)}(z,x_j)dz\right]_{i,j=1}^Ndx_1\cdots dx_N.
\end{align*}
Thus, we obtain the equality of Feller-Markov kernels:
\begin{align}
\mathfrak{Q}_t^{N,(\beta)}\Lambda^*_{N,\beta}\left(y,\cdot\right)=\Lambda^*_{N,\beta}Q_t^{N,(\beta)}\left(y,\cdot\right), \ \forall t \ge 0, \forall y \in W^N_{d,+}.
\end{align}
\end{proof}

\begin{proof}[Proof of Proposition \ref{FellerKernels}]
The claim that the kernels $\Lambda_{N}$ and $\Lambda^*_{N,\beta}$ are positive is due to the fact that, for $x\in W^{N}_{c,+}$ and $y \in W^{N}_{d,+}$:
\begin{align*}
\det\left(x_i^{y_j}\right)_{i,j=1}^N\ge 0.
\end{align*}
This a consequence, after a change of variables, of the well-known fact that the kernel 
\begin{align*}
K(z,w)=e^{zw}
\end{align*}
is strictly totally positive, see \cite{Karlin}.

We now prove that they are normalized to 1. For $x \in \mathring{W}^{N}_{c,+}$ (we extend this to general $x$ below) we can calculate using the Cauchy-Binet or Andreif identity:
\begin{align*}
\sum_{y \in W^{N}_{d,+}}^{}\Lambda_N(x,y)&=\frac{1}{\det\left(x_i^{j-1}\right)_{i,j=1}^N}\sum_{y\in W^{N}_{d,+} }^{}\det\left(y_i^{j-1}\right)_{i,j=1}^N\det \left( \frac{x_i^{y_j}e^{-x_i}}{y_j!}\right)_{i,j=1}^N\\
&=\frac{1}{\det\left(x_i^{j-1}\right)_{i,j=1}^N}\det \left[\sum_{y=0}^{\infty}\frac{x_i^ye^{-x_i}}{y!}y^{j-1}\right]_{i,j=1}^N.
\end{align*}
On the other hand, it is a classical fact that the moments of the Poisson distribution are given in terms of the Touchard polynomials \cite{Touchard} $\mathsf{T}_{\cdot}(\cdot)$:
\begin{align}
\sum_{y=0}^{\infty}\frac{z^ye^{-z}}{y!}y^{i-1}=\mathsf{T}_{i-1}(z)\overset{\textnormal{def}}{=}\sum_{k=0}^{i-1}\begin{Bmatrix}
i-1 \\
k
\end{Bmatrix}z^k.
\end{align}
Here,
\begin{align*}
\begin{Bmatrix}
n \\
k
\end{Bmatrix}= \frac{1}{k!}\sum_{j=0}^{k}(-1)^{k-j}\binom{k}{j}j^n
\end{align*}
is the Stirling number of the second kind. Note that, these polynomials are monic since $\begin{Bmatrix}
n \\
n
\end{Bmatrix}=1$. In particular, we have 
\begin{align*}
\det \left[\sum_{y=0}^{\infty}\frac{x_i^ye^{-x_i}}{y!}y^{j-1}\right]_{i,j=1}^N=\det \left( \mathsf{T}_{j-1}(x_i)\right)_{i,j=1}^N=\det \left( x_i^{j-1}\right)_{i,j=1}^N
\end{align*}
which gives the correct normalization.

The claim that, for any $y \in W^{N}_{d,+}$
\begin{align*}
\int_{x \in W^N_{c,+}}^{}\Lambda^*_{N,\beta}(y,dx)=\frac{1}{\Delta_N(y)}\int_{x \in W^N_{c,+}}^{}\Delta_N(x)\det \left(\frac{x_i^{y_j+\beta+1}e^{-x_i}}{\Gamma(y_j+\beta)}\right)_{i,j=1}^Ndx_1\cdots dx_N=1
\end{align*}
follows by first using the Andreif identity and then the fact that
\begin{align*}
\int_{0}^{\infty}x^k\frac{x^{z+\beta-1}}{\Gamma(z+\beta)}e^{-x}dx=\frac{\Gamma\left(k+z+\beta\right)}{\Gamma\left(z+\beta\right)}=\left(z+\beta+k-1\right)\cdots (z+\beta)
\end{align*}
is a monic polynomial of degree $k$ in $z$.

We now extend $\Lambda_N(x,y)$ to $x \in \partial W^{N}_{c,+}$ by elaborating briefly on the argument from the introduction. We first write it as
\begin{align*}
\Lambda_N(x,y)=\det\left(y_i^{j-1}\right)_{i,j=1}^N\prod_{i=1}^{N}e^{-x_i}\prod_{j=1}^{N}\frac{1}{y_j!}\frac{\det\left(x_i^{y_j}\right)_{i,j=1}^N}{\det\left(x_i^{j-1}\right)_{i,j=1}^N}.
\end{align*}
Now, it suffices to observe that the function
\begin{align*}
\mathfrak{s}_y(x)=\frac{\det\left(x_i^{y_j}\right)_{i,j=1}^N}{\det\left(x_i^{j-1}\right)_{i,j=1}^N}
\end{align*} 
is actually a polynomial (essentially the Schur polynomial) in the variables $(x_1,\cdots,x_N)$ and thus can be extended continuously to $x \in \partial W^{N}_{c,+}$.

Moving on, assume we are given $f\in C_0\left(W^N_{c,+}\right)$ and $\mathfrak{f}\in C_0\left(W^N_{d,+}\right)$. The claim that the function $\left[\Lambda_N \mathfrak{f}\right](\cdot)$ is continuous in $W^N_{c,+}$ is a consequence of the dominated convergence theorem. Observe that, in the case of $\Lambda_{N,\beta}^*$ there is nothing to prove.

Finally, we need to prove that $\left[\Lambda_N \mathfrak{f}\right](\cdot)$ and $\left[\Lambda^*_{N,\beta} f\right](\cdot)$ vanish at infinity. Let $\epsilon>0$ be fixed. We will use the notation $\lesssim$ to mean $\le$ up to a constant independent of $\epsilon$ which might change from line to line.

Pick $R=R(\epsilon)$ such that 
\begin{align*}
|\mathfrak{f}(y)|<\epsilon, \ \forall y \notin W^N_{d,+}\cap \left[0,R(\epsilon)\right]^N,\\
|f(x)|<\epsilon, \ \forall x \notin W^N_{c,+}\cap \left[0,R(\epsilon)\right]^N.
\end{align*}
Then, we can bound:
\begin{align*}
|\left[\Lambda_N\mathfrak{f}\right](x)| &\le \prod_{i=1}^{N}e^{-x_i}\sum_{y \in W^N_{d,+}}^{}\det\left(y_i^{j-1}\right)_{i,j=1}^N\prod_{j=1}^{N}\frac{1}{y_j!}\mathfrak{s}_y(x)|\mathfrak{f}(y)|\\
& \lesssim \prod_{i=1}^{N}e^{-x_i}\sum_{y \in W^N_{d,+}\cap \left[0,R(\epsilon)\right]^N}^{}\det\left(y_i^{j-1}\right)_{i,j=1}^N\prod_{j=1}^{N}\frac{1}{y_j!}\mathfrak{s}_y(x)+\epsilon\\
& \lesssim const(R(\epsilon)) e^{-x_N} x_N^{R(\epsilon)^2}+\epsilon.
\end{align*}
Clearly, taking $x_N$ large enough we obtain
\begin{align*}
const(R(\epsilon)) e^{-x_N} x_N^{R(\epsilon)^2} < \epsilon,
\end{align*}
from which the conclusion for $\Lambda_N$ follows. 

On the other hand:
\begin{align*}
|\left[\Lambda_{N,\beta}^*f\right](y)|&\le \frac{1}{\det\left(y_i^{j-1}\right)_{i,j=1}^N\prod_{j=1}^{N}\Gamma(y_j+\beta)} \int_{x \in W^{N}_{c,+}}^{}\det\left(x_i^{j-1}\right)_{i,j=1}^N\det\left(x_i^{y_j+\beta-1}\right)_{i,j=1}^N|f(x)|dx\\
&\lesssim \frac{1}{\det\left(y_i^{j-1}\right)_{i,j=1}^N\prod_{j=1}^{N}\Gamma(y_j+\beta)}\int_{x \in W^{N}_{c,+}\cap \left[0,R(\epsilon)\right]^N }^{}\det\left(x_i^{j-1}\right)_{i,j=1}^N\det\left(x_i^{y_j+\beta-1}\right)_{i,j=1}^Ndx +\epsilon\\
& \lesssim const(R(\epsilon))\frac{1}{\det\left(y_i^{j-1}\right)_{i,j=1}^N \prod_{j=1}^{N}\Gamma(y_j+\beta)}R(\epsilon)^{N(y_N+\beta-1)}+\epsilon.
\end{align*}
Note that, for any fixed $M$:
\begin{align*}
\frac{M^{y_N}}{\Gamma(y_N+\beta)} \overset{y_N \to \infty}{\longrightarrow} 0,
\end{align*}
from which the conclusion follows.
\end{proof}

\begin{proof}[Proof of Proposition \ref{FellerSemigroups}]
The result that the transition kernel $\mathfrak{q}_t^{N,(\beta)}$ on $W^N_{d,+}$ defined by, for $t>0, x,y \in W^N_{d,+}$:
\begin{align*}
\mathfrak{q}^{N,(\beta)}_t(x,y)=\frac{\Delta_N(y)}{\Delta_N(x)}\det \left(\mathfrak{q}_t^{(\beta)}(x_i,y_j)\right)_{i,j=1}^N
\end{align*}
gives rise to a Feller semigroup on $C_0\left(W^N_{d,+}\right)$ is rather standard. It is an immediate consequence of the following well-known facts (namely the Feller property for $N=1$, see \cite{MicloPatie}):
\begin{align*}
\lim_{t\to 0}\mathfrak{q}^{(\beta)}_t(x,z)&=\delta(x=z), \forall x,z \in \mathbb{Z}_+,\\
\lim_{x \to \infty}\mathfrak{q}_t^{(\beta)}(x,z)&=0 , \forall z \in \mathbb{Z}_+.
\end{align*}
It is important to observe that for all $x \in W^N_{d,+}$ we have $\Delta_N(x)\ge 1$. The reader is referred to Section 5 of \cite{BorodinOlshanski} for a detailed exposition of an entirely analogous example.

To show that $q^{N,(\beta)}_t(x,y)dy$ defined for $(t,x,y)\in (0,\infty)\times \mathring{W}^N_{c,+}\times W^N_{c,+}$ by
\begin{align*}
q^{N,(\beta)}_t(x,dy)=\frac{\Delta_N(y)}{\Delta_N(x)}\det \left(q_t^{(\beta)}(x_i,y_j)\right)_{i,j=1}^Ndy_1\cdots dy_N
\end{align*}
is Feller the situation is a bit more subtle than in the discrete setting. The continuous extension to the boundary $\partial W^{N}_{c,+}$ however is not too hard to establish using the following argument: the singularities coming from $1/\Delta_N(x)$ are cancelled out by the roots of the function
\begin{align*}
(x_1,\dots,x_N)\mapsto\det \left[q^{(\beta)}_t(x_i,y_j)\right]_{i,j=1}^N
\end{align*}
and then one concludes by using the continuity of the partial derivatives $z\mapsto \partial^i_zq^{(\beta)}_t(z,z')$, which can be obtained from the explicit expression for $q^{(\beta)}_t(z,z')$ involving Bessel functions (see for example \cite{SurveyBessel}). However, we shall take a different approach which gives the Feller property (including the continuous extension to $\partial W^N_{c,+}$) in a unified way and avoids the use of explicit formulae.

We will use the connection to the (matrix) Laguerre process which, unlike the system of SDEs (\ref{RepulsiveSDE}), has no singularities and we can appeal to known results. Recall that, the matrix SDE, for $\beta>0$: 
\begin{align}
d\boldsymbol{X}_t=d\boldsymbol{W}_t\sqrt{\frac{\boldsymbol{X}_t^2}{2}}+\sqrt{\frac{\boldsymbol{X}_t^2}{2}}d\boldsymbol{W}_t^{\dagger}+\left[\beta+(N-1)\right]\boldsymbol{I}dt,
\end{align}
has a unique weak solution for any initial condition $\mathbf{X}_0 \in H_+(N)$, where we recall that $H_+(N)$ is the space of non-negative definite Hermitian matrices (with possibly coinciding or zero eigenvalues), see for example \cite{Demni} or Section 3 of \cite{Wishart} (there the case of symmetric positive definite matrices is considered but the same arguments apply to the Hermitian setting), in particular pages 739-741 for the argument for coinciding eigenvalues. Let $\left(\mathcal{W}^{N,(\beta)}(t);t \ge 0\right)$ be the corresponding Markov semigroup. 

By \cite{Demni}, see also Section 3 of \cite{Wishart}, or for general affine processes (the Laguerre/Wishart is a special case) by Section 3 of \cite{AffineProcesses} (again this is for real symmetric matrices but the same arguments give the result in the Hermitian case) this semigroup is actually Feller. 

Note that, the map $f \mapsto f \circ \mathsf{eval}_N$ maps $C_0\left(W^N_{c,+}\right)$ to $C_0\left(H_+(N)\right)$. Now, from the fact that the eigenvalue evolution of $\left(\boldsymbol{X}_t;t\ge0\right)$ is autonomous we obtain that $\forall f :W^N_{c,+}\to \mathbb{R}$ we have:
\begin{align*}
\mathcal{W}^{N,(\beta)}(t)\left(f \circ \mathsf{eval}_N\right)(\mathbf{H}) \ \textnormal{only depends on } \mathbf{H} \textnormal{ through } \mathsf{eval}_N(\mathbf{H}).
\end{align*}
Namely, $\mathsf{eval}_N(\boldsymbol{X}_t)$ only depends on $\mathbf{H}$ through $\mathsf{eval}_N(\boldsymbol{X}_0=\mathbf{H})$. Thus, if $x=\mathsf{eval}_N(\mathbf{H})$ we have:
\begin{align*}
\left[Q_t^{N,(\beta)}f\right](x)=\left[\mathcal{W}^{N,(\beta)}(t)f\circ \mathsf{eval}_N\right](\mathbf{H})=\left[\mathcal{W}^{N,(\beta)}(t)f\circ \mathsf{eval}_N\right](\mathbf{U}^*x\mathbf{U})\ , \forall \mathbf{U}\in \mathbb{U}(N),
\end{align*}
where $\mathbb{U}(N)$ is the group of $N\times N$ unitary matrices. Then, the Feller property of $Q_t^{N,(\beta)}$ is essentially an immediate consequence of the one of $\mathcal{W}^{N,(\beta)}(t)$. 

For example, since $x_n\to \infty \implies \mathbf{U}^*x_n\mathbf{U} \to \infty$ and $\left[\mathcal{W}^{N,(\beta)}(t)f\circ \mathsf{eval}_N\right]\in C_0\left(\overline{H_+(N)}\right)$ for $f \in C_0\left(W^N_{c,+}\right)$, we get:
\begin{align*}
\left[Q_t^{N,(\beta)}f\right](x_n) \to 0 \ \textnormal{ as } \ x_n \to \infty,
\end{align*}
and we can argue likewise for the other conditions.
\end{proof}

\begin{proof}[Proof of Proposition \ref{FellerStationary}]
The proof is completely analogous to the ones of Proposition \ref{FellerKernels} and Proposition \ref{FellerSemigroups}. One now uses the connection to the stationary Laguerre process $\left(\boldsymbol{Y}_t;t\ge0\right)$, solution of the matrix SDE:
\begin{align*}
d\boldsymbol{Y}_t=d\boldsymbol{W}_t\sqrt{\frac{\boldsymbol{Y}_t^2}{2}}+\sqrt{\frac{\boldsymbol{Y}_t^2}{2}}d\boldsymbol{W}_t^{\dagger}+\left(\left[\beta+(N-1)\right]\boldsymbol{I}-\boldsymbol{Y}_t\right)dt,
\end{align*}
with Feller semigroup  $\left(\mathcal{W}_{\mathsf{stat}}^{N,(\beta)}(t);t \ge 0\right)$, see \cite{Demni}, \cite{Wishart}, \cite{AffineProcesses}. As before, we have:
\begin{align*}
\left[K_t^{N,(\beta)}f\right](x)=\left[\mathcal{W}_{\mathsf{stat}}^{N,(\beta)}(t)f\circ \mathsf{eval}_N\right](\mathbf{U}^*x\mathbf{U})\ , \forall \mathbf{U}\in \mathbb{U}(N)
\end{align*}
and we can argue similarly.
\end{proof}

\begin{proof}[Proof of Proposition \ref{StationaryMeasuresProposition}]
The key to proving invariance is reversibility of the one dimensional processes. We calculate, using the fact that $\nu_{\beta}^N(\cdot)$ is supported on $\mathring{W}^N_{c,+}$ and reversibility of $k_t^{(\beta)}(\cdot,\cdot)$ with respect to $\nu_{\beta}(\cdot)$, for $t>0$:
\begin{align*}
\left[\nu_{\beta}^NK_t^{N,(\beta)}\right](dy)&=const_{N,\beta}\times e^{\frac{N(N-1)}{2}t}\Delta_N(y)dy\int_{x \in \mathring{W}^N_{c,+}}^{}\det\left(k_t^{(\beta)}(x_i,y_j)\right)_{i,j=1}^N \Delta_N(x)\prod_{i=1}^{N}\nu_{\beta}(x_i)dx_i\\
&=const_{N,\beta}\times e^{\frac{N(N-1)}{2}t}\Delta_N(y)dy\prod_{i=1}^{N}\nu_{\beta}(y_i)\int_{x \in \mathring{W}^N_{c,+}}^{}\det\left(k_t^{(\beta)}(y_i,x_j)\right)_{i,j=1}^N \Delta_N(x)\prod_{i=1}^{N}dx_i\\
&=const_{N,\beta}\times e^{\frac{N(N-1)}{2}t}\Delta_N(y)\prod_{i=1}^{N}\nu_{\beta}(y_i)e^{-\frac{N(N-1)}{2}t}\Delta_N(y)dy=\nu_{\beta}^N(dy).
\end{align*}
The third equality above is due to the fact that $K_t^{N,(\beta)}$ is Markovian, in particular $K_t^{N,(\beta)}\mathbf{1}=\mathbf{1}$. The case of $\mathfrak{K}^{N,(\beta),\sigma}_t$ and $\eta_{\beta,\sigma}^N$ is completely analogous; just replace integrals by sums. Finally, uniqueness of invariant measures holds for any irreducible Markov chain on a countable state space, see Theorem 1.6 of \cite{AndersonMarkovChains}.
\end{proof}

\begin{proof}[Proof of Theorem \ref{MainResultStationary}]
The proof is entirely analogous to the one of Theorem \ref{MainResult}; one uses the $N=1$ case, proven as Proposition 22 in \cite{MicloPatie}:
\begin{align}
\int_{0}^{\infty}k_t^{(\beta)}(x,z)\frac{(\sigma z)^ye^{-\sigma z}}{y!}dz=\sum_{w=0}^{\infty}\frac{(\sigma x)^we^{-\sigma x}}{w!}\mathfrak{k}_t^{(\beta),\sigma}(w,y), \ t>0, x\in \mathbb{R}_+,y \in \mathbb{Z}_+,
\end{align}
and the Andreif identity.
\end{proof}

\section{Connection to the Young bouquet and the z-measures on partitions}\label{ConnectiontoYoungBouquetSection}

This section is independent of the rest of the paper. The aim is to explain how this paper, and in particular Theorem \ref{MainResultStationary}, is related to a series of works by Borodin and Olshanski \cite{BorodinOlshanskiMarkovonPartitions}, \cite{BorodinOlshanskiMeixner}, \cite{BorodinOlshanskiYoungBouquet}, \cite{BorodinOlshanskiThoma}, \cite{OlshanskiLaguerreDiffusions}, \cite{OlshanskiThomaDiffusion}, \cite{OlshanskiRepresentationRing}, \cite{OlshanskiSymmetricFunctions}. We assume that the reader is somewhat familiar with the basics of graded graphs and projective systems, see for example Section 2 of \cite{BorodinOlshanskiYoungBouquet}, that we partially follow, for a nice exposition.

\subsection{The Young bouquet and its boundary}\label{YoungBouquetSubsection}

\subsubsection{The Young graph}
We first introduce the Young graph $\mathbb{Y}$, a distinguished graded graph that is associated to the branching of irreducible representations of the chain of symmetric groups $S(1)\subset S(2) \subset \cdots \subset S(N) \subset S(N+1)\subset \cdots$, see for example \cite{BorodinOlshanskiBook}. 
\begin{defn}
The vertices of the Young graph are given by partitions or equivalently Young diagrams $\mathbb{Y}$ (we use the same notation as for the graph). The $n^{\textnormal{th}}$ level of the graph is given by $\mathbb{Y}_n$, the set of Young diagrams with $n$ boxes (we also write $\mathbb{Y}_0=\emptyset$, the empty diagram, for the root of the graph). Two vertices (diagrams) on consecutive levels are joined by an edge iff they differ by a box. 
\end{defn}

Let $\textnormal{dim}(\lambda)$ denote the number of paths in the Young graph (from the root) ending at vertex $\lambda$ (equivalently the number of standard Young tableaux of shape $\lambda$, see \cite{BorodinOlshanskiYoungBouquet}, \cite{BorodinOlshanskiBook}). Then, we can define the following Markov kernel ${}^{\mathbb{Y}}\Lambda^{m+1}_m$ from $\mathbb{Y}_{m+1}$ to $\mathbb{Y}_m$:
\begin{align*}
{}^{\mathbb{Y}}\Lambda^{m+1}_m\left(\nu, \lambda\right)=\frac{\textnormal{dim}(\lambda)}{\textnormal{dim}(\nu)}\textbf{1}\left(\lambda \subset \nu \right), \ \lambda\in \mathbb{Y}_m, \nu \in \mathbb{Y}_{m+1}.
\end{align*}
Here, $\lambda \subset \nu$ means that the diagram $\lambda$ is included in $\nu$, in this particular case $\nu$ is obtained from $\lambda$ by adding a box. More generally, for $n>m$ we define:
\begin{align*}
{}^{\mathbb{Y}}\Lambda^{n}_m={}^{\mathbb{Y}}\Lambda^{n}_{n-1}{}^{\mathbb{Y}}\Lambda^{n-1}_{n-2}\cdots{}^{\mathbb{Y}}\Lambda^{m+1}_m.
\end{align*}
We say that a sequence of probability measures $\{\mu_n\}_{n=1}^\infty$ on $\{\mathbb{Y}_n \}_{n=1}^{\infty}$ is coherent if:
\begin{align*}
\mu_{m+1}{}^{\mathbb{Y}}\Lambda^{m+1}_m=\mu_m, \ \forall{m}\ge 1.
\end{align*}
Then, the boundary of the Young graph, namely the set of extremal coherent sequences of probability measures on $\mathbb{Y}$, is in bijection with the Thoma simplex $\Omega$ defined as follows (see Section 3 of \cite{BorodinOlshanskiYoungBouquet} for more about this remarkable result):
\begin{defn}
The Thoma simplex $\Omega$ is the subspace of $\mathbb{R}_+^{\infty}\times \mathbb{R}_+^{\infty}$ formed by the couple of sequences $\alpha=(\alpha_i), \beta=(\beta_i)$ such that :
\begin{align*}
\alpha_1\ge \alpha_2 \ge \cdots \ge 0, \beta_1 \ge \beta_2 \ge \cdots \ge 0,\sum_{i=1}^{\infty}\alpha_i+\sum_{i=1}^{\infty}\beta_i \le 1.
\end{align*}
\end{defn}
Moreover, there exist explicit (see Section 3 of \cite{BorodinOlshanskiYoungBouquet}) Markov kernels ${}^{\mathbb{Y}}\Lambda^{\infty}_m$ from $\Omega$ to $\mathbb{Y}_m$ satisfying:
\begin{align*}
{}^{\mathbb{Y}}\Lambda^{\infty}_{m+1}{}^{\mathbb{Y}}\Lambda^{m+1}_m={}^{\mathbb{Y}}\Lambda^{\infty}_{m}, \ \forall m \ge 1.
\end{align*}

\subsubsection{The Young bouquet}
\begin{defn}
The Young bouquet is the poset $\left(\mathbb{YB},<\right)$ defined as follows: $\mathbb{YB}$ is obtained from the direct product $\mathbb{Y} \times \mathbb{R}_+$ by glueing together all points $\left(\nu,0\right)$ to a single point $\left(\emptyset,0\right)$. An element $(\nu,r) \in \mathbb{YB}$ is smaller than $(\nu',r')\in \mathbb{YB}$ if $r<r'$ and $\nu \subset \nu'$. We write $|(\nu,r)|=r$ and call this the level of $(\nu,r)$. 
\end{defn}

Let $\mathbb{YB}_r$ be the subset of elements with level $r$ and consider the stratification $\mathbb{YB}=\underset{r\ge 0}{\sqcup}\mathbb{YB}_r$ (observe also that we can identify each $\mathbb{YB}_r$ with $\mathbb{Y}$). Now, for any pair $r'>r>0$ consider the following Markov kernel ${}^{\mathbb{YB}}\Lambda^{r'}_{r}$ from $\mathbb{YB}_{r'}$ to $\mathbb{YB}_r$:
\begin{align*}
{}^{\mathbb{YB}}\Lambda^{r'}_{r}(\nu,\lambda)=\left(1-\frac{r}{r'}\right)^{|\nu|-|\lambda|}\left(\frac{r}{r'}\right)^{|\lambda|}\frac{|\nu|!}{\left(|\nu|-|\lambda|\right)!|\lambda|!}{}^{\mathbb{Y}}\Lambda^{|\nu|}_{|\lambda|}(\nu,\lambda).
\end{align*}
These Markov kernels satisfy the compatibility relations, see Section 3 of \cite{BorodinOlshanskiYoungBouquet}:
\begin{align*}
{}^{\mathbb{YB}}\Lambda^{r^{''}}_{r'}{}^{\mathbb{YB}}\Lambda^{r'}_{r}={}^{\mathbb{YB}}\Lambda^{r{''}}_{r}, \ r^{''}>r'>r>0,
\end{align*}
and thus the Young bouquet forms a projective system. Its boundary, see Section 3 of \cite{BorodinOlshanskiYoungBouquet}, is in bijection with the Thoma cone $\tilde{\Omega}$ defined as follows:
\begin{defn}
The Thoma cone $\tilde{\Omega}$ is the subspace of $\mathbb{R}_+^{\infty} \times \mathbb{R}_+^{\infty} \times \mathbb{R}_+$, formed by triples $\omega=\left(\alpha,\beta,\delta\right)$ so that $\alpha=(\alpha_i)$ and $\beta=(\beta_i)$ satisfy:
\begin{align*}
\alpha_1\ge \alpha_2 \ge \cdots \ge 0, \beta_1 \ge \beta_2 \ge \cdots \ge 0,\sum_{i=1}^{\infty}\alpha_i+\sum_{i=1}^{\infty}\beta_i \le \delta.
\end{align*}
\end{defn}
 
Clearly, we can identify the Thoma simplex $\Omega$ with the subset of the Thoma cone consisting of $\omega \in \tilde{\Omega}$ with $\delta(\omega)=1$.

Moreover, there exist explicit Markov kernels ${}^{\mathbb{YB}}\Lambda^{\infty}_{r}$ from $\tilde{\Omega}$ to $\mathbb{YB}_r=\mathbb{Y}$ satisfying the compatibility relations (see Section 3 of \cite{BorodinOlshanskiYoungBouquet}):
\begin{align}
{}^{\mathbb{YB}}\Lambda^{\infty}_{r'}{}^{\mathbb{YB}}\Lambda^{r'}_{r}={}^{\mathbb{YB}}\Lambda^{\infty}_{r}, \ r'>r>0.
\end{align}

We require a final definition:
\begin{defn}\label{SubsetThomaCone}
For $x \in W^N_{c,+}$ we define $\omega_x=(\alpha(\omega_x),0,\delta(\omega_x)) \in \tilde{\Omega}$ as follows: 
\begin{align*}
\alpha(\omega_x)=(x_N\ge x_{N-1}\ge \cdots \ge x_2\ge x_1),
\end{align*}
$\alpha_l(\omega_x)=0$ for $l>N$, $\beta_i(\omega_x) \equiv 0$ and $\delta(\omega_x)=\sum_{i=1}^N x_i$. 
\end{defn}

With all these preliminaries in place, the following proposition explains the connection of the Markov kernel $\Lambda_{N,\sigma}$ from $W^N_{c,+}$ to $W^N_{d,+}$ defined in (\ref{MarkovKernel}) with the Young bouquet. 
\begin{prop}\label{RelationBoundaryKernels}
We have that, for $x \in W^N_{c,+}$:
\begin{align*}
{}^{\mathbb{YB}}\Lambda^{\infty}_{r}\left(\omega_x,\cdot\right) \textnormal{ is supported on } \mathbb{Y}(N).
\end{align*}
Moreover, under the bijection between $W^N_{d,+}$ and $\mathbb{Y}(N)$, we have the following equality of probability measures:
\begin{align}
{}^{\mathbb{YB}}\Lambda^{\infty}_{r}\left(\omega_x,\cdot\right)=\Lambda_{N,r}\left(x, \cdot\right), \ x \in W^N_{c,+}.
\end{align}
\end{prop}

\begin{proof}
This is a direct consequence of the explicit formula for ${}^{\mathbb{YB}}\Lambda^{\infty}_{r}\left(\omega_x,\cdot\right)$ from Section $3$ of \cite{BorodinOlshanskiYoungBouquet} along with the explicit formula for $\textnormal{dim}(\lambda)$ from Section 1 of \cite{BorodinOlshanskiMarkovonPartitions}.
\end{proof}

\subsection{The z-measures and the Meixner ensemble}\label{zMeasuresSubsection}
We now define, the celebrated z-measures on partitions, see \cite{BorodinOlshanskiZMeasures},\cite{BorodinOlshanskiMeixner},\cite{BorodinOlshanskiYoungBouquet}. They are a distinguished special case of Okounkov's Schur measures \cite{SchurMeasures}.

Let $\left(z\right)_{\lambda}$ be the generalized Pochhammer symbol:
\begin{align*}
\left(z\right)_{\lambda}=\prod_{i=1}^{l(\lambda)}\left(z-i+1\right)_{\lambda_i}, \ z \in \mathbb{C}, \lambda \in \mathbb{Y},
\end{align*}
where for $n \in \mathbb{N}$, $(z)_n=z(z+1)\cdots(z+n-1)$, $(z)_0=1$.

Consider the following conditions on a pair of parameters $(z,z')$. We call any pair $(z,z')$ satisfying one of the three conditions below admissible.
\begin{itemize}
\item (Principal series) The numbers $z$ and $z'$ are not real and moreover complex conjugate to each other.
\item (Complementary series) Both $z$ and $z'$ are real and contained in an interval $(m,m+1)$ for some $m \in \mathbb{Z}$.
\item (Degenerate series) One of $z,z'$ is a non-zero integer, say $z=N$, while $z'$ has the same sign and $|z'|>|z|-1=N-1$. 
\end{itemize}

\begin{defn}
The z-measure on partitions $M_{z,z'}^r$ with admissible parameters $(z,z')$ and additional parameter $r>0$ is defined as follows:
\begin{align}
M_{z,z'}^r\left(\lambda\right)=\left(1+r\right)^{-zz'}\left(\frac{r}{1+r}\right)^{|\lambda|} \left(z\right)_{\lambda} \left(z'\right)_{\lambda}\left(\frac{\textnormal{dim}(\lambda)}{|\lambda|!}\right)^2, \ \lambda \in \mathbb{Y}.
\end{align}
\end{defn}

A key fact about the z-measures is that they are consistent on $\mathbb{YB}$, see \cite{BorodinOlshanskiYoungBouquet}:
\begin{align*}
M_{z,z'}^{r'}{}^{\mathbb{YB}}\Lambda^{r'}_{r}=M_{z,z'}^{r}, \ \forall r'>r>0
\end{align*}
and thus (see \cite{BorodinOlshanskiYoungBouquet}) give rise to a unique probability measure $M^{\infty}_{z,z'}$ on $\tilde{\Omega}$ so that:
\begin{align*}
M^{\infty}_{z,z'}{}^{\mathbb{YB}}\Lambda^{\infty}_{r}=M^{r}_{z,z'}, \ \forall r>0.
\end{align*}
The z-measure with parameters in the degenerate series $z=N, z'=N+\beta-1$ and $r>0$ coincides (under the bijection between $W^N_{d,+}$ and $\mathbb{Y}(N)$) with the Meixner ensemble $\eta_{\beta,r}^N(\cdot)$, see for example \cite{BorodinOlshanskiMeixner}. We then have the following result:

\begin{prop}\label{BoundaryZmeasure}
The boundary $z$-measure $M_{N,N+\beta-1}^{\infty}$ on the Thoma cone $\tilde{\Omega}$ coincides with the Laguerre ensemble $\nu_{\beta}^N$. More precisely, if we consider $\omega_x=(\alpha(\omega_x),0,\delta(\omega_x)) \in \tilde{\Omega}$ as in Definition \ref{SubsetThomaCone} with $x\in W^N_{c,+}$ picked according to $\nu_{\beta}^N$, then $\omega_x$ has law $M_{N,N+\beta-1}^{\infty}$.
\end{prop}

\begin{proof}
This is a consequence of Proposition \ref{ExactRelationLaguerreMeixner} and Proposition \ref{RelationBoundaryKernels} along with the discussion above.
\end{proof}
\subsection{Markov processes for z-measures}\label{MarkovProcessesonBoundarySubsection}
A construction of a Markov process on $\tilde{\Omega}$, possessing the Feller property with additional desirable features including a determinantal structure, that preserves the non-degenerate z-measures was obtained in \cite{BorodinOlshanskiThoma}, see also \cite{BorodinOlshanskiMarkovonPartitions},\cite{BorodinOlshanskiDiffusionsPartitions}, \cite{OlshanskiLaguerreDiffusions} for previous studies. The strategy follows the method of intertwiners of Borodin and Olshanski introduced in \cite{BorodinOlshanski}, for constructing Feller processes on boundaries of projective systems, see \cite{BorodinOlshanski},\cite{Cuenca},\cite{HuaPickrell} for applications of this method. 

The statement of the result in \cite{BorodinOlshanskiThoma} goes as follows: For non-degenerate parameters $(z,z')$ there exists a unique Feller-Markov semigroup $\left(T^{\infty}_{z,z'}(t)\right)_{t\ge 0}$ on $\tilde{\Omega}$, with $M_{z,z'}^{\infty}$ as its unique invariant measure, satisfying (and in fact characterized through) the intertwining relation:
\begin{align}\label{nondegenerateintertwining}
T^{\infty}_{z,z'}(t){}^{\mathbb{YB}}\Lambda^{\infty}_{r}={}^{\mathbb{YB}}\Lambda^{\infty}_{r}T^{r}_{z,z'}(t), \ \forall t \ge 0, r>0,
\end{align}
where $\left(T^{r}_{z,z'}(t)\right)_{t\ge 0}$ is the semigroup of a certain Markov jump process on $\mathbb{Y}$ (see \cite{BorodinOlshanskiThoma} for its definition), which has $M_{z,z'}^r$ as its unique invariant probability measure. 

Then, the authors go on to identify the generator $\mathfrak{D}^{\infty}_{z,z'}$ of the abstract semigroup $\left(T^{\infty}_{z,z'}(t)\right)_{t\ge 0}$ by its action on a certain algebra of functions on $\tilde{\Omega}$, see \cite{BorodinOlshanskiThoma}, and in a subsequent paper \cite{OlshanskiThomaDiffusion} it is shown that $\left(T^{\infty}_{z,z'}(t)\right)_{t\ge 0}$ gives rise to a Markov process with continuous sample paths. In all these works, heavy use is made of symmetric function theory. The key role is played by the Laguerre and Meixner symmetric functions introduced and studied by Olshanski in \cite{OlshanskiSymmetricFunctions}.

In fact, due to Propositions \ref{RelationBoundaryKernels} and \ref{BoundaryZmeasure} above we can interpret Theorem \ref{MainResultStationary} in this paper as the analogue of (\ref{nondegenerateintertwining}) for the degenerate series of parameters $(z,z')$, thus completing the picture for the whole range of admissible parameter values.

Finally, we should mention that an intertwining relation between a diffusion generator and that of a Markov jump process is proven in Section 9 of \cite{OlshanskiRepresentationRing}. The motivation behind this study is the analogous problem of constructing dynamics for the zw-measures on the Gelfand-Tsetlin graph \cite{BorodinOlshanskiBoundary}, \cite{BorodinOlshanski}. Again, heavy use is made of symmetric functions and a key role is played by the Jacobi and Hahn orthogonal polynomials.

We finish with a number of remarks.

\begin{rmk}
It would be interesting to understand whether the intertwining (\ref{MainIntertwining2}), going in the opposite direction, has any meaning as well in this framework of consistent dynamics on projective systems.
\end{rmk}

\begin{rmk}
It would also be interesting to see whether it is possible to obtain the results for the non-degenerate case from the ones for the degenerate one, by some kind of analytic continuation, as was done in \cite{BorodinOlshanskiMarkovonPartitions}, \cite{OlshanskiSymmetricFunctions}.
\end{rmk}

\begin{rmk}
Theorem \ref{MainResultStationary} can also be used to obtain relations between the multivariate Meixner and Laguerre polynomials as in \cite{BorodinOlshanskiThoma}, \cite{OlshanskiSymmetricFunctions}. Observe that, this is going in the opposite direction of the arguments in \cite{BorodinOlshanskiThoma} which go from information on the symmetric functions to obtain results for the Markov semigroups.
\end{rmk}

\bigskip
\noindent
{\sc School of Mathematics, University of Bristol, U.K.}\newline
\href{mailto:T.Assiotis@bristol.ac.uk}{\small T.Assiotis@bristol.ac.uk}

\end{document}